\documentclass[11pt,dvips,twoside,letterpaper]{article}
\usepackage{pslatex}
\usepackage{fancyhdr}
\usepackage{graphicx}
\usepackage{geometry}
\RequirePackage[latin1]{inputenc} \RequirePackage[T1]{fontenc}

\def\figurename{Figure} 
\makeatletter
\renewcommand{\fnum@figure}[1]{\figurename~\thefigure.}
\makeatother

\def\tablename{Table} 
\makeatletter
\renewcommand{\fnum@table}[1]{\tablename~\thetable.}
\makeatother
\usepackage{color}
                [2000/03/29 v1.0m Standard LaTeX package]

\usepackage{amsmath}
\usepackage{amssymb}
\usepackage{amsfonts}
\usepackage{amsthm,amscd}
\newtheorem{theorem}{Theorem}[section]

\newtheorem{proposition}[theorem]{Proposition}
\theoremstyle{definition}

\newtheorem{definition}[theorem]{Definition}

\theoremstyle{remark}
\newtheorem{remark}[theorem]{Remark}

\numberwithin{equation}{section}

\def\P{\mathbb P}
\def\R{\mathbb R}
\def\E{\mathbb E}

\def\E{\mathbb E}

\def\cal{\mathcal}

\begin{document}

\author{Auguste Aman\thanks{augusteaman5@yahoo.fr,\ Corresponding author.}\;; Abouo Elouaflin\thanks{elabouo@yahoo.fr}\, and Modeste N'zi\thanks{modeste.nzi@univ-cocody.ci} \\
UFR de Mathématiques et Informatique\\
Université de Cocody, Côte d'Ivoire\\
22 BP 582 Abidjan 22}
\date{}
\title{Reflected generalized BSDEs with random time and applications}
\maketitle
\begin{abstract}
In this paper, we aim to study solutions of reflected generalized BSDEs, involving the integral with respect to a continuous process, which is the local
time of the diffusion on the boundary. We consider both a finite random terminal and a infinite horizon. In both case, we establish an existence and uniqueness result. Next, as an application, we get an American pricing option in infinite horizon and we give a probabilistic formula for the viscosity solution of an obstacle problem for elliptic PDEs with a nonlinear Neumann boundary condition.

\end{abstract}


\section{Introduction}
Generalized backward stochastic differential equations ( for short
GBSDEs ) has been considered by Pardoux and Zhang \cite{PZ} as an
extension of nonlinear BSDE which involves an integral with
respect to an increasing process. They provide probabilistic
representation of viscosity solutions of both parabolic and
elliptic PDE with Neumann boundary condition. Let us mention that
the now well- known theory of nonlinear backward stochastic
differential equations was formulated by Pardoux and Peng
\cite{PP}. Since, they have found several fields of applications.
Namely, we refer to Pardoux \cite{P1} and \cite{P2}, El Karoui et
al \cite{EPQ}, Cvitanic and Ma \cite {CM} for the applications in
mathematical finance and to Hamad\`{e}ne, Lepeltier \cite{HL} for
the applications in stochastic control and stochastic games. On
other hand, El Karoui et al \cite{EKPP} have considered reflected
BSDEs where the ``reflection'' keeps the solution above a given
stochastic process called an obstacle. In this setting, many
others results have been established in the literature, among
others, we note the work of Hamad\`{e}ne et al \cite{HL1,HLA}, Cvitanic and Ma \cite{CM1}%
, Hamad\`{e}ne and Ouknine \cite{HO}. Recently, Ren and Xia \cite{Ral} give a probabilistic formula
for the viscosity solution of an obstacle problem for parabolic PDEs with a
nonlinear Neumann boundary condition. They use the connection with such PDEs and the reflected GBSDEs.
We notice that above result is with deterministic horizon and Lipschitz condition on the coefficients.

To fill the gap, this paper is devoted to derive existence and uniqueness result to reflected GBSDEs with random terminal time which may be infinite and non Lipchitz coefficients. In application, we give an optimal stopping time problem related to American pricing option, using a infinite horizon reflected GBSDEs. With a finite random time one, we derive a probabilistic formula for the viscosity solution of an obstacle problem for elliptic PDEs with a nonlinear Neumann boundary condition. The rest of this paper is organized as follows. We precise our problem in section 2. Section 3 and Section 4 are devoted to the main results. In section 5, we give as an application, the connection with American option pricing and an obstacle problem for a elliptic PDEs with nonlinear Neumann boundary condition.
\section{Formulation of the problem}
Let $(\Omega , {\cal F},{\Bbb P})$ be a complete probability space
and $\left( W_{t},{\cal F}_{t}\right) _{t\geq 0}$ be a
$d$-dimensional Wiener process defined on it. $\left\{ {\cal
F}_{t}\right\} $ denotes is natural filtration augmented with all
${\Bbb P}$-null sets of ${\cal F}$ and $\mathcal{F}_{\infty}=\bigcup_{t\geq 0}\mathcal{F}_t$. Let us consider the
following objects:
\begin{description}
\item  $\left( \text{\bf A1}\right)\left\{\begin{array}{lll} (i)\
\tau$ {\mbox is a} $ \mathcal{F}_{t}$-$ \mbox{stopping time}.\
\\
\cr (ii)\ \left(G_{t}\right) _{t\geq 0} \mbox{is a continuous real
valued increasing} \mathcal{F}_{t}$-$\mbox{progressively
measurable}\\ \mbox{process verifying}\ G_{0}=0
\end{array}\right.
$
\item  $\left( \text{\bf A2}\right) $ $f$ and $g$ are $\mbox{I\hspace{-.15em}R}$-%
values measurable functions defined respectively on $\Omega \times %
\R_{+}\times \mbox{I\hspace{-.15em}R}\times %
\mbox{I\hspace{-.15em}R}^{d}$ and $\Omega \times \mbox{I\hspace{-.15em}R}%
_{+}\times \mbox{I\hspace{-.15em}R}$ such that there are constants $\alpha \in %
\R,\ \beta <0,\ K
>0, \ \lambda>2|\alpha|+K^{2}$ and
 $\mu>2|\beta|$ and $[1,+\infty )$-valued process
$\{\varphi _{t},\ \psi _{t}\}_{t\leq 0}$ verifying\\
$
\begin{array}{lll}
\\
(i)\ \forall t,\forall z,y\longmapsto (f(t,y,z),g(t,y))\text{ is continuous}%
\cr\cr(ii)\ \left( \omega ,t\right) \longmapsto (f(\omega
,t,y,z),g(\omega
,t,y))\text{ is }\mathcal{F}_{t}$-$\text{progressively measurable} \\
\cr(iii)\ \forall t,\forall y,\forall \left( z,z^{\prime }\right)
,\ \
|f(t,y,z)-f(t,y,z^{\prime })|\leq K |z-z^{\prime }| \\
\cr(iv)\ \forall t,\forall z,\forall (y,y^{\prime }),\ \left(
y-y^{\prime }\right) \left( f(t,y,z)-f(t,y^{\prime },z)\right)
\leq \alpha |y-y^{\prime
}|^{2} \\
\cr(v)\ \forall t,\,\forall (y,y^{\prime }),\ \left(
y-y^{\prime }\right) \left( g(t,y)-g(t,y^{\prime })\right) \leq
\beta |y-y^{\prime }|^{2}
\\
\cr(vi)\ \forall t,\forall y,\forall z,\ |f(t,y,z)|\leq \varphi
_{t}+K(|y|+|z|),\ \ |g(t,y)|\leq \psi _{t}+K|y| \\
\cr(vii)\text{ }\E \left[ \int_{0}^{\tau}e^{\lambda s+\mu
G(s)}[\varphi (s)^{2}ds + \psi (s)^{2}]dG_{s}\right]<\infty .
\end{array}
$ \item  $\left( \text{\bf A3}\right)\xi $ is a
$\mathcal{F}_{\tau}$-measurable variable such that $\E(e^{\lambda
\tau +\mu G(\tau)}|\xi |^{2})<+\infty $

\item  $\left( \text{\bf A4}\right) $ $\left( S_{t}\right) _{t\geq
0}$ is a continuous progressively measurable real-valued process
satisfying:\newline $
\begin{array}{lll}
\left( i\right) &  & \E\left( \sup_{0\leq t\leq
\tau}e^{\lambda t+\mu G_t}(S_{t}^{+})^{2}\right) <+\infty \\
&  &  \\
\left( ii\right) &  & S_{\tau}\leq \xi \text{ }\mathbb{P}\text{ a.s.}
\end{array}
$
\end{description}

Let $\left(\tau ,\xi ,f,g,S\right)$ be the data satisfying the previous conditions. We want to construct an adapted processes $(Y_t , Z_t,K_t)_{t\geq 0}$ solution
of the reflected GBSDE
\begin{eqnarray}
-dY_t={\bf 1}_{t\leq \tau}f(t,Y_t,Z_t)dt+{\bf 1}_{t\leq \tau}g(t,Y_t)dG_t+dK_t-Z_tdW_t,\;\;\; Y_{\tau}=\xi\label{BSDE}
\end{eqnarray}
or equivalently
\begin{eqnarray}
Y_{t\wedge\tau}=\xi +\int_{t\wedge \tau}^{\tau}f(t,Y_t,Z_t)dt+\int_{t\wedge \tau}^{\tau}g(t,Y_t)dG_t-\int_{t\wedge \tau}^{\tau}Z_tdW_t+K_{\tau}-K_{t\wedge\tau}.
\end{eqnarray}
Let us first recall that a solution to the equation $(\ref{BSDE})$ is a triplet of progressively measurable
processes $\left(Y_{t},Z_{t},K_{t}\right)_{t\geq 0}$ with values in $\R\times \R^{d}\times \R$ such that
\begin{enumerate}
\item  $Y$ is a continuous process, $\P$-a.s., for each $T,\, t\mapsto Z_t$ belongs to $L^{2}((0,T);\R^{d})$ and\\
$t\mapsto (f(t,Y_t,Z_t), g(t,Y_t))\in L^{1}((0,T);\R)\times L^{1}((0,T);\R)$;
\item For all $t\geq \tau\;\; a.s., \, Y_{t}=\xi,\;\; Z_{t}=0,\;\; K_{t}=K_{\tau}$;
\item for each nonnegative real $T$,\, $\forall t\in [0,T]$,
\begin{eqnarray*}
Y_{t}=Y_{T\wedge \tau}+\int_{t\wedge\tau}^{T\wedge\tau}f(s,Y_{s},Z_{s})ds+\int_{t\wedge\tau}^{T\wedge\tau}g(s,Y_{s})dG_{s}
-\int_{t\wedge\tau}^{T\wedge\tau}Z_{s}dW_{s}+K_{T\wedge\tau}-K_{t\wedge\tau}.
\end{eqnarray*}
\item$\ Y_{t}\geq S_{t},\, t\geq 0$
\item $\E\left( \sup_{0\leq t\leq \tau
}e^{\lambda t+\mu G(t)}\left| Y_{t}\right| ^{2}+\int_{0}^{\tau
}e^{\lambda s+\mu G(s)}\left[ \left( \left| Y_{s}\right|
^{2}+\left| Z_{s}\right| ^{2}\right)
ds+\left| Y_{s}\right| ^{2}dG_{s}\right] \right) <+\infty$
\item $K$ is a non-decreasing process such that $K_{0}=0$ and $\int_{0}^{\tau}\left( Y_{t}-S_{t}\right) dK_{t}=0$\, a.s.
\end{enumerate}

\section{ Reflected GBSDEs with finite random terminal time}
The aim of this section is to prove the first main result of this paper, concerning the existence and uniqueness result
for reflected GBSDEs $(\ref{BSDE})$ when the random time $\tau$ is suppose to be finite.
\begin{theorem}
Assume that $({\bf A1})$-$({\bf A4})$ hold.
Moreover if the obstacle process  $(S_{t})_{t\geq 0}$ is the Itô process in
the form $\displaystyle{dS_{t}=m_{t}{\bf
1}_{\left[0,\tau\right]}dt+v_{t}{\bf
1}_{\left[0,\tau\right]}dW_{t}}$,\\ with
$\displaystyle{\E\left(\int_{0}^{\tau }e^{\lambda s+\mu
G(s)}\left(|m_{s}|^{2}+|v_{s}|^{2}\right)ds\right)
<+\infty}$ . Then there exists a unique triple $\left(
Y,Z,K\right) $ solution of reflected GBSDE $(\ref{BSDE})$.
\end{theorem}

\begin{proof}
We adopt this strategy for the proof.

{\bf  Existence.} For each integer $n$, let us denote $\xi_n=\E(\xi|\mathcal{F}_{n})$ and consider the data
\newline
$(\xi_n,{\bf 1}_{[0,\tau]}f,{\bf 1}_{[0,\tau]}g,S_{.\wedge\tau})$. Under $(\bf A1)$-$(\bf A4)$, one can show, using the same argument as in \cite{Ral} that there exists a unique process $(\overline{Y}^n,\overline{Z}^n,\overline{K}^n)$, solution of the classical (deterministic terminal time) reflected GBSDE
\begin{eqnarray}
\overline{Y}_{t}^{n}&=&\xi_n+\int_{t}^{n}{\bf 1}_{[0,\tau]}f(s,\overline{Y}_{s}^{n},\overline{Z}_{s}^{n}) ds
+\int_{t}^{n}{\bf 1}_{[0,\tau]}g(s,\overline{Y}_{s}^{n})dG_{s}\nonumber\\
&&-\int_{t}^{n}\overline{Z}_{s}^{n}dW_{s}+\overline{K}^{n}_{n}-\overline{K}^{n}_{t},\, 0\leq t\leq n,\label{BSDEf}
\end{eqnarray}
satisfying:
\begin{description}
\item$
\begin{array}{l}
\overline{Y}^n_t\geq S_t \,\, \mbox{and}\,\, \int^{n\wedge\tau}_0(\overline{Y}^{n}_t-S_t)d\overline{K}^{n}_t=0.
\end{array}
$
\end{description}

Since $\xi$ belongs to $L^{2}(\mathcal{F}_\tau)$, there exists a process $(\eta_t)_{t\geq 0}$ in $M^{2}(0, \tau; \R^d )$
such that
\begin{eqnarray*}
\xi=\E[\xi]+\int_0^{\tau}\eta_sdW_s
\end{eqnarray*}
and, we define $(\overline{Y}^n, \overline{Z}^n, \overline{K}^n)$ on the whole time axis by setting:
\begin{eqnarray*}
\forall\, t>n,\, \overline{Y}^n_t=\E(\xi|\mathcal{F}_{t})=\xi_t\;\;\;\;\; \overline{Z}^n_t=\eta_t{\bf 1}_{[0,\tau]}\,\,\, \mbox{and}\;\;\;\;\overline{K}_t^n=\overline{K}_{n}^n.
\end{eqnarray*}
In the sequel, we consider the process $(Y^n,Z^n,K^n)$ defined by: $Y^n_t = \overline{Y}^n_{t\wedge \tau},\; Z^{n}_t= Z^{n}_{t\wedge\tau}$ and $K^n_t = \overline{K}^n_{t\wedge \tau}$.

The rest of the proof will be split in several steps and, $C$ denotes a positive constant which may vary from one
line to another.

{\bf Step 1}{\it : A priori estimates uniform in }$n.$\newline
First, there exists a constant $C>0$ such that for all $s\ge 0$,
\begin{eqnarray}
&&\E\left( \sup_{0\leq t\leq \tau }e^{\lambda t+\mu G_t}\left|
Y_{t}^{n}\right| ^{2}+\int_{0}^{\tau }e^{\lambda s+\mu G_s}\left[( \left|
Y_{s}^{n}\right| ^{2}+\left| Z_{s}^{n}\right|^{2})ds+
\left|Y_{s}^{n}\right| ^{2}dG_{s}\right]+ |K_\tau^n|^{2}\right) \nonumber\\\label{est}\\
&\leq &C\E\left( e^{\lambda \tau +\mu G_{\tau}}\left| \xi \right|
^{2}+\int_{0}^{\tau }e^{\lambda s+\mu G_s}\left[ \varphi^{2}(s)
ds+\psi^{2}(s) dG_{s}\right]+\sup_{0\leq t\leq \tau }e^{\lambda t+\mu G_t}\left| \left(
S_{t}^{{}}\right) ^{+}\right| ^{2}\right).\nonumber
\end{eqnarray}
Indeed, for any arbitrarily small $\varepsilon > 0$ and any $\rho < 1$
arbitrarily close to one, there exists a constant $C>0$ such that for all $s > 0,\, y\in \R,\, z\in\R^d$,
\begin{eqnarray*}
2\langle y,f(s,y,z)\rangle &\leq& (2\alpha +\rho^{-1}K^{2}+\varepsilon )|y|^{2}+\rho|z|^{2}+c\varphi^{2}(s), \\
2\langle y,g(s,y)\rangle &\leq &(2\beta +\varepsilon)|y|^{2}+c\psi^{2}(s).
\end{eqnarray*}
From these and Itô's formula, we deduce that for any arbitrarily small $\delta>0$
\begin{eqnarray}
&&\E\left(e^{\lambda t+\mu G_{t}}| Y_{t}^{n}| ^{2}+\int_{t\wedge\tau}^{\tau }e^{\lambda s+\mu
G_s}[(\bar{\lambda}| Y_{s}^{n}| ^{2}+\bar{\rho}| Z_{s}^{n}| ^{2})ds+\bar{\mu }| Y_{s}^{n}| ^{2}dG_{s}]\right)  \nonumber\\
&\leq &\E\left(e^{\lambda \tau +\mu G_{\tau }}|\xi|^{2}+2c\int_{t\wedge\tau}^{\tau}e^{\lambda
s+\mu G_s}\left[\varphi^{2}(s)ds+\psi^{2}(s)dG_{s}\right]+2\int_{t\wedge\tau}^{\tau}e^{\lambda s+\mu G_s}\langle S_s,dK_{s}^{n}\rangle\right)\nonumber\\
&\leq&\E\left(e^{\lambda \tau +\mu G_{\tau }}|\xi|^{2}+2c\int_{t\wedge\tau}^{\tau}e^{\lambda
s+\mu G_s}\left[\varphi^{2}(s)ds+\psi^{2}(s)dG_{s}\right] \right.\nonumber \\
&&\left.+\delta^{-1}\sup_{0\leq t\leq \tau}e^{\lambda s+\mu G_s}(S_s^{+})^{2}+\delta(K_{\tau}^{n}-K^{n}_{t})^{2}\right)
,  \label{D1}
\end{eqnarray}
where $\bar{\lambda }=\lambda -2\alpha -\rho^{-1}K^{2}-\varepsilon,\; \bar{\rho}=1-\rho$ and $%
\bar{\mu }=\mu -2\beta -\varepsilon$. We may choose $\varepsilon$ and $\rho$ such that $\bar{\lambda }> 0, \,
\bar{\rho}> 0$ and $\bar{\mu}> 0$. From the reflected GBSDE $(\ref{BSDEf})$, estimate $(\ref{D1})$ and for every $\lambda ^{\prime }$ such that $0<\lambda ^{\prime}<\min \left( \lambda ,\mu \right)$, we have
\begin{eqnarray*}
&&\delta {\Bbb E}\left| K_{\tau }^{n}-K_{t}^{n}\right|
^{2} \\
&\leq &\delta {\Bbb E}\left( \left| Y_{t}^{n}\right|
^{2}+\left| \xi\right| ^{2}+(\lambda ^{\prime })^{-1}
\int_{t\wedge\tau}^{\tau }e^{\lambda ^{\prime }s}\left( \varphi
^{2}(s)+\left| Y_{s}^{n}\right| ^{2}+\left| Z_{s}^{n}\right| ^{2}\right)
ds\right. \\
&&\left. +(\lambda ^{\prime })^{-1}\int_{t\wedge\tau}^{\tau}e^{\lambda ^{\prime }G_s}\left( \psi ^{2}(s)+\left| Y_{s}^{n}\right|
^{2}\right) dG_{s}\right)\\
&\leq &\delta {\Bbb E}\left( e^{\lambda t+\mu G_t}\left| Y_{t}^{n}\right| ^{2}+e^{\lambda\tau+\mu
G_{\tau}}\left|\xi\right|^{2}\right) \\
&&+\delta (\lambda ^{\prime} )^{-1}{\Bbb E}\left( \int_{t\wedge\tau}^{\tau }e^{\lambda
s+\mu G(s)}\left[ \left| Y_{s}^{n}\right| ^{2}+ \varphi ^{2}(s)+\left| Z_{s}^{n}\right| ^{2}\right]ds\right) \\
&&+\delta(\lambda ^{\prime} )^{-1} {\Bbb E}\left( \int_{t\wedge\tau}^{\tau }e^{\lambda
s+\mu G(s)}(|Y^{n}_s|^{2}+\psi ^{2}(s))dG_{s}\right).
\end{eqnarray*}
Chosen $\delta $ small enough such that $1-\delta(\lambda ^{\prime} )^{-1} >0,\; \bar{\bar{
\lambda }}=\bar{\lambda }-\delta(\lambda ^{\prime} )^{-1} >0,\;\bar{\bar{\rho}}=\bar{\rho}-\delta(\lambda ^{\prime} )^{-1} >0$ and $
\bar{\bar{\mu }}=\bar{\mu }-\delta(\lambda ^{\prime} )^{-1} >0,$ we get
\begin{eqnarray*}
&&\E\left[(1-\delta(\lambda ^{\prime} )^{-1})e^{\lambda t+\mu G_t}| Y_{t}^{n}|^{2}+
\int_{t\wedge\tau}^{\tau }e^{\lambda s+\mu G_s}\left([\bar{\bar{\lambda }}|Y_{s}^{n}|^{2}+\bar{\bar{\rho}}| Z_{s}^{n}| ^{2}]ds+\bar{\bar{\mu }}|Y_{s}^{n}|^{2}dG_{s}\right)\right]\\ &\leq &C\E\left( e^{\lambda \tau+\mu G_{\tau}}|\xi|^{2}+\int_{t\wedge\tau}^{\tau}e^{\lambda s+\mu G_s}[\varphi^{2}(s)ds+\psi^{2}(s)dG_{s}]+\sup_{0\leq t\leq \tau }e^{\lambda t+\mu G(t)}( S_{t}^{+})^{2}\right) .
\end{eqnarray*}
Therefore, the result follows by using Burkhölder-Davis-Gundy inequality.

{\bf Step 2}{\it : Convergence of the sequence }$\left(
Y^{n},Z^{n},K^{n}\right) .$\newline
For $m > n$, let us set $\Delta Y_t=Y^{m}_t-Y^n_t,\; \Delta Z_t=Z^m_t- Z^n_t,\;  \Delta K_t=K^m_t- K^n_t$. In view of \eqref{BSDEf}, we get
\begin{eqnarray*}
-d(\Delta Y)_{t}&=&(f(s,Y^n_s,Z^n_s)-f(s,Y^m_s,Z^m_s))ds+(g(s,Y^n_s)-g(s,Y^m_s))dG_s\\
&&-\Delta Z_tdW_t+d(\Delta K)_{s},
\end{eqnarray*}
from which, Itô's formula and above assumptions yield
\begin{eqnarray}
&&e^{\lambda t+\mu G_{t}}|\Delta Y_{t}|^{2}+\int_{t\wedge\tau}^{m\wedge\tau}e^{\lambda s+\mu G_s}[(\bar{\lambda}|\Delta Y_{s}|^{2}+\bar{\rho}|\Delta Z_s|)ds+\bar{\mu}|\Delta Y_s|^{2}dG_{s}]\nonumber\\
&\leq &e^{\lambda m+\mu G_{m}}|\Delta Y_{m}|^{2}+\int_{t\wedge\tau}^{m\wedge\tau}\langle\Delta Y_s,d(\Delta K_s)\rangle-2\int_{t\wedge\tau}^{m\wedge\tau}e^{\lambda s+\mu G_s}\langle
\Delta Y_{s},\Delta Z_{s}dW_{s}\rangle.\label{estconvergence}
\end{eqnarray}
Furthermore, since one can show that
\begin{eqnarray*}
\int_{t\wedge\tau}^{m\wedge\tau}\langle\Delta Y_s,d(\Delta K_s)\rangle\leq 0,
\end{eqnarray*}
by taking expectation in both side of \eqref{estconvergence} and using Burkhölder-Davis-Gundy inequality, we get
\begin{eqnarray*}
&&\E\left(\sup_{0\leq t\leq \tau}e^{\lambda t+\mu G_t}\left|\Delta Y_t\right|^{2}+\int_{0}^{\tau}e^{\lambda s+\mu G_s}[(\bar{\lambda}|\Delta Y_{s}|^{2}+\bar{\rho}|\Delta Z_s|)ds+\bar{\mu}|\Delta Y_s|^{2}dG_{s}]\right)\nonumber\\
&\leq &\E\left(e^{\lambda(m\wedge\tau)+\mu G_{m\wedge\tau}}|\Delta Y_{m}|^{2}\right).
\end{eqnarray*}
But, since $\Delta Y_m=\xi_{m\wedge\tau}-\xi_{n\wedge\tau}$,
\begin{eqnarray*}
\E\left(\sup_{0\leq t\leq \tau}e^{\lambda t+\mu G_t}\left|\Delta Y_t\right|^{2}+\int_{0}^{\tau}e^{\lambda s+\mu G_s}[(\bar{\lambda}|\Delta Y_{s}|^{2}+\bar{\rho}|\Delta Z_s|)ds+\bar{\mu}|\Delta Y_s|^{2}dG_{s}]\right)
\end{eqnarray*}
tends to zero as $n,m$ goes to infinity. Therefore, $(Y^n,Z^n)$ is a Cauchy sequence and converges to $(Y,Z)$.
In virtue of $(\ref{BSDEf})$, the convergence of $Y^n,\ Z^n$ (for a subsequence), the continuity of $f$ and $g$ and
\begin{itemize}
\item $\sup_{n\geq0}|f(s,Y^n_s,Z_s)|\leq f_s+K\left\{(\sup_{n\geq0}|Y_s^n|)+\|Z_s\|\right\}$,
\item $\sup_{n\geq0}|\phi(s,Y^n_s)|\leq \phi_s+K\left\{(\sup_{n\geq0}|Y^n_s|)\right\}$,
\item $\mathbb{E}\int_0^T|f(s,Y^n_s,Z^n_s)-f(s,Y^n_s,Z_s)|^2ds\leq C\mathbb{E}\int_0^T\|Z^n_s-Z_s\|^2ds$,
\end{itemize}
there exists a  process $K$ such that for all
$t\in[0,T]$
$$\E\left| K_{t}^{n}-K_{t}^{{}}\right| ^{2}\longrightarrow 0$$
as $n$ goes to infinity.

{\bf Step 4} {\it The limit process }$\left( Y,Z,K\right) $ {\it solves our
reflected GBSDE }$\left( \tau ,\xi ,f,g,S\right).$

Taking the limit in BSDE $(\ref{BSDEf})$, we get $\mathbb{P}$-a.s. for any $T>0$,
$$Y_t=\xi+\int_t^{\tau\wedge T} f(s,Y_s,Z_s)ds+\int_t^{\tau\wedge T} g(s,Y_s)dG_s+K_{\tau\wedge T}-K_t-\int_t^{\tau\wedge T} Z_sdW_s,\, \forall t\in[0,T\wedge\tau]$$
and for all $t\geq\tau$, $Y_t=\xi,\; Z_t=0,\; K_t=K_\tau$. Moreover, since $(Y^{n}_t,K^{n}_t)_{0\leq t\leq T}$ tends to $(Y_t,K_t)_{0\leq t\leq T}$ in probability, the measure $dK^n$ converges to $dK$ in probability, so that $\int_{0}^{n\wedge\tau}(Y_{s}^{n}-S_{s})dK_{s}^{n}\rightarrow \int_{0}^{\tau}(Y_{s}-S_{s})dK_{s}$ in probability as $n\rightarrow \infty$. Hence, $
\int_{0}^{\tau }\left( Y_{s}-S_{s}\right) dK_{s}= 0$.

{\bf  Uniqueness}\newline
Let $(Y_{t},Z_{t},K_{t})$ and $(Y'_{t},Z'_{t},K'_{t})$ be two solutions of the
reflected GBSDE $(\ref{BSDE})$, and $(\bar{Y}_t,\bar{Z}_t,\bar{K}_t)=(Y_t-Y'_t,Z_t-Z'_t,K_t-K'_t)$. It follows from Itô's formula,
the assumptions $(iii)$, $(iv)$ and $(v)$ of $({\bf A2})$ that
\begin{eqnarray*}
&&e^{\lambda (t\wedge \tau )+\mu G_{t\wedge \tau }}| \bar{Y}_{t\wedge
\tau }|^{2}+\int_{t\wedge \tau }^{T\wedge \tau }e^{\lambda s+\mu
G_{s}}[\lambda |\bar{Y}_{s}|^{2}ds+\mu|\bar{Y}_{s}|^{2}dG_{s}+|\bar{Z}_{s}|^{2}ds] \\
&\leq&e^{\lambda (T\wedge \tau )+\mu G_{T\wedge \tau}}\left|\bar{Y}_{T\wedge
\tau }\right| ^{2}+2\int_{t\wedge \tau }^{T\wedge \tau }e^{\lambda s+\mu
G(s)}[\alpha|\bar{Y}_{s}|^{2}+K|\bar{Y}_s|\times|\bar{Z}_{s}|^{2}]ds \\
&&2\beta\int_{t\wedge \tau }^{T\wedge \tau }e^{\lambda s+\mu G(s)}|\bar{Y}_{s}|^{2}dG_{s}
-2\int_{t\wedge \tau }^{T\wedge \tau }e^{\lambda s+\mu G(s)}\langle
\bar{Y}_{s},\bar{Z}_{s}dW_{s}\rangle.
\end{eqnarray*}
Hence, with $\rho < 1, \bar{\lambda}=\lambda-2\alpha -\rho^{-1}K^{2}>0, \bar{\mu}=\mu-2\beta>0$,
\begin{eqnarray*}
&&\E\left(e^{\lambda (t\wedge \tau )+\mu G_{t\wedge \tau }}| \bar{Y}_{t\wedge
\tau }|^{2}+\int_{t\wedge \tau }^{T\wedge \tau }e^{\lambda s+\mu
G_{s}}[\lambda |\bar{Y}_{s}|^{2}ds+\mu|\bar{Y}_{s}|^{2}dG_{s}+(1-\rho)|\bar{Z}_{s}|^{2}ds]\right)\\
&\leq&\E\left(e^{\lambda (T\wedge \tau )+\mu G_{T\wedge \tau}}|\bar{Y}_{T\wedge\tau}|^{2}\right),
\end{eqnarray*}
and consequently, letting $T\rightarrow \infty $, dominated convergence theorem yields
\begin{eqnarray*}
\E\left( e^{\lambda (t\wedge \tau )+\mu G(t\wedge \tau )}\left|\bar{Y}_{t\wedge \tau }\right| ^{2}\right) =0.
\end{eqnarray*}
Then for all $t$, $ \bar{Y}_{t\wedge \tau }=0$ and $\bar{Z}_{t\wedge \tau }=0.$ Moreover, since
\begin{eqnarray*}
\bar{K}_{t\wedge\tau} &=&\bar{Y}_{0}-\bar{Y}_{t\wedge\tau}-\int_{0}^{t\wedge \tau
}f(s,Y_{s},Z_{s}^{{}})-f(s,Y_{s}^{\prime },Z_{s}^{\prime })ds \\
&&-\int_{0}^{t\wedge \tau }g(s,Y_{s})-g(s,Y_{s}^{\prime
})dG_{s}+\int_{0}^{t\wedge \tau }\bar{Z}_{s}dW_{s},
\end{eqnarray*}
$\bar{K}_{t\wedge\tau}=0$ for all $t$.
\end{proof}

\section{Infinite horizon reflected GBSDEs}
 In this section, we study the following infinite horizon reflected GBSDE:
\begin{eqnarray}
Y_t=\xi+\int_{t}^{\infty}f(s,Y_s,Z_s)ds+\int_{t}^{\infty}g(s,Y_s)ds-\int_{t}^{\infty}Z_sdW_s+K_{\infty}-K_t,\,\, 0\leq t\leq \infty.\label{BSDEinfh}
\end{eqnarray}
Let us introduce some spaces which our discussion will be carried on.
\begin{eqnarray*}
\mathcal{S}^{2}=\left\{\varphi_t,\, 0 \leq t\leq\infty,\, \mbox{ is an}\, \mathcal{F}_t\mbox{-adapted process such that}, \E\left(\sup_{0\leq t\leq \infty}|\varphi_t|^{2}\right)<\infty\right\},
\end{eqnarray*}
\begin{eqnarray*}
\mathcal{H}^{2}=\left\{\varphi_t,\, 0 \leq t\leq\infty,\, \mbox{ is an}\, \mathcal{F}_t\mbox{-adapted process such that},\,
 \E\left(\int_{0}^{\infty}|\varphi_t|^{2}dt\right)<\infty\right\},
\end{eqnarray*}
Throughout the paper, we propose the following assumptions:
\begin{description}
\item $({\bf A2'})$ $f:\Omega\times[0, \infty)\times\R\times\R^d\rightarrow\R$ and $g:\Omega\times[0, \infty)\times\R\rightarrow\R$ measurable mappings and three positives deterministic processes $u,\, v$ and $v'$ verifying
\begin{eqnarray}
\int_{0}^{\infty}[(v_t+v'^{2}_t)dt+u_tdG_t]<+\infty.\label{estispecial}
\end{eqnarray}
such that\\\\
$\begin{array}{l}
(i)\;|f(t,y,z)-f(t,y',z')|\leq v_t|y-y'|+v'_t\|z-z'\|,\\\\
(ii)\; |g(t,y)-g(t,y')|\leq u_t|y-y'|\\\\
(iii)\, \langle y-y',g(t,y)-g(t,y')\rangle\leq \beta|y-y'|^{2}\\\\
(iii)\; |f(t,y,z)|\leq \varphi_t+K(|y|+\|z\|),\;\;  |g(t,y)|\leq \psi_t+K|y|\\\\
(iv)\, \E\left(\int_{0}^{\infty}\varphi_t^{2}ds+\psi_t^{2}dG_t\right)< \infty.
\end{array}$
\item $({\bf A3'})$ a terminal value $\xi\in L^{2}(\Omega,\mathcal{F}_{\infty},\P)$
\item $({\bf A4'})$ The barrier $(S_t,\; t\geq 0)$ is a continuous progressively
measurable real-valued process such that\\\\
$\begin{array}{l}
(i)\, \E[\sup_{t\geq 0}(S_{t}^{+})^2]< \infty\\\\
(ii)\, \limsup_{t\nearrow \infty} S_t \leq \xi,\;  a.s.
\end{array}$
\end{description}
With all the above preparations, we have
\begin{definition}
A solution to reflected GBSDE associated with the data $(\xi,f,g,S)$ is a triple $(Y_t, Z_t, K_t)$ of $\mathcal{F}_t$ progressively measurable
processes such that $(\ref{BSDEinfh})$ holds and
\begin{description}
\item $(i)\,\, Y\in\mathcal{S}^{2},\;\; Z\in\mathcal{H}^{2},\;\;  K_{\infty}\in L^{2}$;
\item $(ii)\;\; Y_t\geq S_t,\;\; t\geq \infty$;
\item $(iii)\,\,  K_t$ is continuous and increasing, $K_0 = 0$, and $\int_{0}^{\infty}(Y_t-S_t)dK_t = 0$.
\end{description}
\end{definition}
Our approach to solve above reflected GBSDEs with infinite horizon is to use the snell envelope theory connected to the contraction method. For this, we consider first the special case that is the function $f$ and $g$ do not depend on $(Y,Z)$ such that
\begin{eqnarray}
\E\left(\int_{0}^{\infty}|f(t)|^{2}dt+\int_{0}^{\infty}|g(t)|^{2}dG_t\right)<\infty.\label{A5}
\end{eqnarray}
More precisely we have the following reflected GBSDE:
\begin{eqnarray}
Y_t = \xi+ \int_{t}^{\infty}f(s)ds+\int_{t}^{\infty}g(s)dG_s-\int_{t}^{\infty}Z_sdW_s+K_{\infty}-K_t, \;\;t\in[0, \infty].\label{BSDEinfhp}
\end{eqnarray}
\begin{proposition}
Assume that $({\bf A3'})$,$({\bf A4'})$ and $\eqref{A5}$ hold.
Then reflected GBSDE $(\ref{BSDEinfhp})$ associated with $(\xi,f,g,S)$ has a unique solution $(Y, Z, K)$.
\end{proposition}
\begin{proof}
Let $(F_t)_{0\leq t\leq \infty}$ be the process defined as follows:
\begin{eqnarray*}
F_t=\int^{t}_0 f(s)ds+\int^{t}_0 g(s)dG_s+S_{t}{\bf 1}_{t<\infty}+\xi{\bf 1}_{t=\infty}.
\end{eqnarray*}
Then for $t<\infty$, $F$ is continuous $\mathcal{F}_t$-adapted process and $\sup_{0\leq t\leq \infty}F_t\in L^{2}(\Omega,\mathcal{F}_{\infty})$. So, the Snell envelope of $F$ is the smallest continuous supermartingale which dominates the process $F$ and it is given by:
\begin{eqnarray*}
\mathcal{S}_t(F)=ess\sup_{\nu\in\mathcal{K}_t}\E\left(F_{\nu}|\mathcal{F}_t\right),
\end{eqnarray*}
where $\mathcal{K}_t$ is the set of all $\mathcal{F}_s$-stopping times taking values in $[t, \infty]$. Then, we have
\begin{eqnarray*}
\E\left(\sup_{0\leq t\leq \infty}[\mathcal{S}_t(F)]^2\right)<\infty
\end{eqnarray*}
hence $(\mathcal{S}_t(F))_{0\leq t\leq \infty}$ is of class [D]. Therefore, it has the
following Doob-Meyer decomposition:
\begin{eqnarray*}
\mathcal{S}_t(F)=\E\left(\xi+\int_{0}^{\infty}f(t)ds+\int_{0}^{\infty}g(t)dG_t+K_{\infty}|\mathcal{F}_t\right)-K_t
\end{eqnarray*}
where $(K_t)_{0\leq t\leq\infty}$ is an $\mathcal{F}_t$-adapted continuous non-decreasing process such that $K_0=0$. By the theory of Snell envelope (see Ren and Hu, \cite{RH}) we have $\E(K_\infty)^2<\infty$. Therefore we derive
\begin{eqnarray*}
\E\left[\sup_{0\leq t\leq\infty}\left|\E\left(\xi+\int_{0}^{\infty}f(t)ds+\int_{0}^{\infty}g(t)dG_t+K_{\infty}|\mathcal{F}_t\right)\right|^2\right]<\infty
\end{eqnarray*}
and then, through the martingale representation there exists a continuous uniformly integrable process $(Z_s)_{0\leq s\leq \infty}$such that
\begin{eqnarray*}
M_t&=&\E\left(\xi+\int_{0}^{\infty}f(t)ds+\int_{0}^{\infty}g(t)dG_t+K_{\infty}|\mathcal{F}_t\right)\\
&=&M_0+\int_0^t Z_s dW_s.
\end{eqnarray*}
Now let us set
\begin{eqnarray*}
Y_t=ess\sup_{\nu\in\mathcal{K}_t}\E\left[\int^{\nu}_t f(s)ds+\int^{\nu}_t g(s)dG_s+S_{\nu}{\bf 1}_{\nu<\infty}+\xi{\bf 1}_{\nu=\infty}\right].
\end{eqnarray*}
Then
\begin{eqnarray*}
Y_t+ \int_0^{t}f(s)ds+\int_0^{t}g(s)dG_s&=&\mathcal{S}_t(F)\\
&=& M_t-K_t
\end{eqnarray*}
henceforth, we have
\begin{eqnarray*}
Y_t+ \int_0^{\infty}f(s)ds+\int_0^{\infty}g(s)dG_s=\xi+\int_0^{\infty}f(s)ds+\int_0^{\infty}g(s)dG_s+\int_0^{t}Z_sdW_s-K_t.
\end{eqnarray*}
So, we obtain
\begin{eqnarray*}
Y_t=\xi+ \int_t^{\infty}f(s)ds+\int_t^{\infty}g(s)dG_s+K_{\infty}-K_t-\int_t^{\infty}Z_sdW_s, \; 0\leq t\leq \infty.
\end{eqnarray*}
Since, $Y_t+ \int_0^{t}f(s)ds+\int_0^{t}g(s)dG_s=\mathcal{S}_t(F)$ and $\mathcal{S}_t(F)\geq F_t=\int^{t}_0 f(s)ds+\int^{t}_0 g(s)dG_s+S_{t}{\bf 1}_{t<\infty}+\xi{\bf 1}_{t=\infty}$, then $Y_t\geq S_t$.

Finally, use again the theory of Snell envelope, we know $\int_0^\infty(\mathcal{S}_t(F)-F_t)dK_t=0$ i.e.
\begin{eqnarray*}
\int_0^{\infty}(Y_t-S_t)dK_t=\int_0^\infty(\mathcal{S}_t(F)-F_t)dK_t=0.
\end{eqnarray*}
Therefore, the triple $(Y, Z, K)$ satisfies the
reflected GBSDE $(\ref{BSDEinfhp})$ and properties $(i)$-$(iii)$ above.

Let us prove uniqueness. If $(Y', Z', K')$ is another solution of the reflected
generalized GBSDE $(\ref{BSDEinfhp})$ associated with $(\xi,f,g,S)$ satisfying properties $(i)$-$(iii)$ above, define
$\bar{Y}= Y-Y',\, \bar{Z}= Z-Z'$, and $K = K-K'$. Using Itô's formula to $|\bar{Y}_t|^2$,
\begin{eqnarray}
|\bar{Y}_t|^2+\int_{t}^{\infty}|\bar{Z}_s|^{2}ds=2\int_{t}^{\infty}\bar{Y}_s d\bar{K}_s-2\int_{t}^{\infty}\bar{Y}_s\bar{Z}_s dW_s,\label{uniq}
\end{eqnarray}
by the integrable conditions $(i)$-$(iii)$ and Burkholder-Davis-Gundy's inequality, we have
\begin{eqnarray*}
\E\left(|\bar{Y}_t|^2+\int_{t}^{\infty}|\bar{Z}_s|^{2}ds\right)=2\E\left(\int_{t}^{\infty}\bar{Y}_s d\bar{K}_s\right)\leq 0.
\end{eqnarray*}
So $\E(\bar{Y}_t)= 0$ a.s. for all $t\in[0,\infty]$ and $\E\left(\int_{t}^{\infty}|\bar{Z}_s|^{2}ds\right)= 0$. Then $|\bar{Y}_t|^2 =|\bar{Z}_t|^2=0$ a.s., so
that $Y= Y'$ by the continuity of $\bar{Y}_t$ and $Z=Z'$. Finally, it is easy to get $K = K'$ a.s.
\end{proof}

We now establish the main result of this section.
\begin{theorem}
Assume that $({\bf A2'})$, $({\bf A3'})$ and $({\bf A4'})$ hold. Then
the reflected GBSDE $(\ref{BSDEinfh})$ associated with $(\xi,f,g,S)$ has a unique solution
$(Y, Z, K)$.
\end{theorem}
\begin{proof}
We first prove the uniqueness. Let $(Y, Z, K)$ and $(Y', Z', K')$ be two solutions
of the reflected GBSDE $(\ref{BSDEinfh})$ associated with $(\xi,f,g, S)$. By use the same
notation as in Proposition 3.1 and applying Itô's formula to $|\bar{Y}_t|^{2},$ we have
\begin{eqnarray*}
|\bar{Y}_t|^{2}+\int^{\infty}_{t}|\bar{Z}_s|^{2}ds &=& 2 \int^{\infty}_t\bar{Y}_s(f(s,Y_s,Z_s)-f(s,Y'_s,Z'_s))ds+2 \int^{\infty}_t\bar{Y}_s(g(s,Y_s)-g(s,Y'_s))dG_s\\
&&+2\int^{\infty}_t\bar{Y}_s d\bar{K}_s-2 \int^{\infty}_t\bar{Y}_s\bar{Z}dW_s.
\end{eqnarray*}
Then
\begin{eqnarray}
\E\left(|\bar{Y}_t|^{2}+\int^{\infty}_{t}|\bar{Z}_s|^{2}ds\right) &\leq& 2 \E\int^{\infty}_t|\bar{Y}_s|(v_s|\bar{Y}_s|+v'_s|\bar{Z}_s|)ds\nonumber\\
&&+2\beta\E\int^{\infty}_t|\bar{Y}_s|^{2}dG_s+2\E\int^{\infty}_t\bar{Y}_s d\bar{K}_s\nonumber\\
&\leq&\frac{1}{2}\E\int^{\infty}_t|\bar{Z}_s|^2ds+ \E\int^{\infty}_t(2v_s+2v'^{2}_s)|\bar{Y}_s|^{2}ds\label{estuniqueness}
\end{eqnarray}
From Gronwall's lemma we obtain $\E|\bar{Y}_t|^{2}=0$ for all $t\in [0,\infty]$. Then $|\bar{Y}_t|^2= 0$
as., so $Y = Y'$ by the continuity of $\bar{Y}_t$. Now, going back to $(\ref{estuniqueness})$, we have
\begin{eqnarray*}
\E\int^{\infty}_0|\bar{Z}_s|^2 ds\leq  \E\sup_{0\leq t\leq \infty}|\bar{Y}_s|^{2}\int^{\infty}_0(2v_s+2v'^{2}_s)ds,
\end{eqnarray*}
so $\E\int^{\infty}_0|\bar{Z}_s|^2 ds= 0$. Then it is easy to get $K_t = K'_t$.

At last, we prove the existence of (\ref{BSDEinfh}). It is divided into two steps.

{\bf Step 1.} Assume $(\int^{\infty}_0 v_sds+u_sdG_s)^2+\int^{\infty}_0 v'^{2}_s ds < \frac{1}{24}$.\newline
Let us denote $\mathcal{D}=\mathcal{S}^2\times\mathcal{H}^2$ and $\|(Y, Z)\|_{\mathcal{D}}= \|Y\|^{2}_{\mathcal{S}^2}+\|Z\|^{2}_{\mathcal{H}^2}$. We define a mapping $\Psi:\mathcal{D}\rightarrow\mathcal{D}$ as follows: for any $(U, V )\in\mathcal{D},\; (Y, Z ) = \Psi(U, V)$ is a
element of $\mathcal{D}$ such that $(Y,Z,K)$ is a unique solution to reflected GBSDE associated with $(\xi, f (s, U_s, V_s),g(s,U_s),S)$. Similarly we define $(Y', Z') = \Psi(U',V ')$ for $(U', V')\in\mathcal{D}$ and set $\bar{U} = U-U',\; \bar{V}= V-V',\; \bar{Y}= Y-Y',\;
\bar{Z}= Z-Z',\;\bar{K} = K-K', \; \bar{f}=f (s, U_s, V_s)-f(s, U'_s, V'_s)$ and $\bar{g}=g (s, U_s)-g(s, U'_s)$. From above we have
\begin{eqnarray*}
Y_t&=&ess\sup_{\nu\in\mathcal{K}_t}\E\left(\int^{\nu}_t f(s,U_s,V_s)ds+\int^{\nu}_t g(s,U_s)dG_s+S_{\nu}{\bf 1}_{\nu<\infty}+\xi{\bf 1}_{\nu=\infty}|\mathcal{F}_t\right),\\
Y'_t&=&ess\sup_{\nu\in\mathcal{K}_t}\E\left(\int^{\nu}_t f(s,U'_s,V'_s)ds+\int^{\nu}_t g(s,U'_s)dG_s+S_{\nu}{\bf 1}_{\nu<\infty}+\xi{\bf 1}_{\nu=\infty}|\mathcal{F}_t\right).
\end{eqnarray*}
Then
\begin{eqnarray*}
|\bar{Y}_{t}|&\leq& ess\sup_{\nu\in\mathcal{K}_{t}}\E\left(\int^{\nu}_t|\bar{f}(s)|ds+\int^{\nu}_t |\bar{g}(s)|dG_s|\mathcal{F}_t\right)\\
&\leq&\E\left(\int^{\infty}_{0} |\bar{f}(s)|ds+\int^{\infty}_{0}|\bar{g}(s)|dG_s|\mathcal{F}_{t}\right)
\end{eqnarray*}
which provides
\begin{eqnarray*}
\E\left(\sup_{0\leq t\leq \infty}|\bar{Y}_t|^{2}\right)&\leq&\E\left[\sup_{0\leq t\leq \infty}\E\left(\int^{\infty}_0|\bar{f}(s)|ds+\int^{\infty}_0 |\bar{g}(s)|dG_s|\mathcal{F}_t\right)^{2}\right] \\
&\leq&4\E\left(\int^{\infty}_0 |\bar{f}(s)|ds+\int^{\infty}_0 |\bar{g}(s)|dG_s\right)^{2}
\end{eqnarray*}
by Doob's inequality. Using Itô's formula to $|\bar{Y}_t|^2$, we get
\begin{eqnarray*}
|\bar{Y}_t|^{2}+\int^{\infty}_{t}|\bar{Z}_s|^{2}ds&=& 2 \int^{\infty}_t\bar{Y}_s\bar{f}(s)ds+2 \int^{\infty}_t\bar{Y}_s\bar{g}(s)ds+2\int^{\infty}_t\bar{Y}_s d\bar{K}_s-2 \int^{\infty}_t\bar{Y}_s\bar{Z}_sdW_s.\\
&\leq& 2 \int^{\infty}_t\bar{Y}_s\bar{f}(s)ds-2 \int^{\infty}_t\bar{Y}_s\bar{Z}_sdW_s.
\end{eqnarray*}
Then
\begin{eqnarray*}
\E\left(\int^{\infty}_{t}|\bar{Z}_s|^{2}ds\right) &\leq& 2\int^{\infty}_0\bar{Y}_s\bar{f}(s)ds\\
 &\leq&\E\left(\sup_{0\leq t\leq \infty}|Y_t|^{2}\right)+\E\left(\int^{\infty}_0|\bar{f}(s)|ds\right)^{2}.\\
 &\leq &4\E\left(\int^{\infty}_0[|\bar{f}(s)|ds+|\bar{g}(s)|dG_s]\right)^{2}+\E\left(\int^{\infty}_0|\bar{f}(s)|ds\right)^{2}.
\end{eqnarray*}
From $({\bf A2'})$ we get
\begin{eqnarray*}
&&\E\left(\int^{\infty}_0[|\bar{f}(s)|ds+|\bar{g}(s)|dG_s]\right)^{2}+\E\left(\int^{\infty}_0|\bar{f}(s)|ds\right)^{2}\\
&\leq& \E\left(\int_{0}^{\infty}(v_s|\bar{U}_s|+v'_s|\bar{V}_s|)ds+u_s|\bar{U}_s|dG_s\right)^{2}\\
&\leq&4\left[\left(\int^{\infty}_0v_sds+u_sdG_s\right)^2+\int^{\infty}_0v'^{2}ds\right]\|(\bar{U},\bar{V})\|_{\mathcal{D}}.
\end{eqnarray*}
At last, we have
\begin{eqnarray}
\|(\bar{Y},\bar{Z})\|_{\mathcal{D}}\leq 24\left[\left(\int^{\infty}_0v_sds+u_sdG_s\right)^2+\int^{\infty}_0v'^{2}ds\right]\|(\bar{U},\bar{V})\|_{\mathcal{D}}.
\end{eqnarray}
From the inequality $(\int^{\infty}_0 v_sds+u_sdG_s)^2+\int^{\infty}_0 v'^{2}_s ds < \frac{1}{24}$ we infer that $\Psi$ is a strict
contraction and has a unique fixed point, which is a unique solution of the
reflected GBSDE $(\ref{BSDEinfh})$.

{\bf Step 2.} For the general case i.e \eqref{estispecial}, there exists $T_0 > 0$ such that
\begin{eqnarray*}
\left(\int^{\infty}_{T_0} v_sds+u_sdG_s\right)^2+\int^{\infty}_{T_0} v'^{2}_s ds < \frac{1}{24}.
\end{eqnarray*}
From Step 1 we know that the reflected GBSDE
\begin{eqnarray}
\widehat{Y}_t&=&\xi+\int_{t}^{\infty}{\bf 1}_{\{s\geq T_0\}}f(s,\widehat{Y}_s,\widehat{Z}_s)ds+\int_{t}^{\infty}{\bf 1}_{\{s\geq T_0\}}g(s,\widehat{Y}_s)ds\nonumber\\
&&-\int_{t}^{\infty}\widehat{Z}_sdW_s+\widehat{K}_{\infty}-\widehat{K}_t,\,\, 0\leq t\leq \infty,\label{BSDEinfh1}
\end{eqnarray}
has a unique solution $(\widehat{Y},\widehat{Z},\widehat{K})$. Then
we consider the reflected GBSDE
\begin{eqnarray}
\widetilde{Y}_t&=&\xi+\int_{t}^{T_0}f(s,\widetilde{Y}_s,\widetilde{Z}_s)ds+\int_{t}^{T_0}g(s,\widetilde{Y}_s)ds\nonumber\\
&&-\int_{t}^{T_0}\widetilde{Z}_sdW_s+\widetilde{K}_{T_0}-\widetilde{K}_t,\,\, 0\leq t\leq T_{0}.\label{BSDEinfh2}
\end{eqnarray}
It follows from \cite{Ral}, the existence of a unique solution $(\widetilde{Y},\widetilde{Z},\widetilde{K})$ of reflected GBSDE $(\ref{BSDEinfh2})$.

Let us set
\begin{eqnarray*}
Y_t=
\left\{
\begin{array}{l}
\widetilde{Y}_t,\;\;\; t\in[0,T_0],\\\\
\widehat{Y}_t,\;\;\; t\in[T_0,\infty],
\end{array}\right.
\;\;\;
Z_t=
\left\{
\begin{array}{l}
\widetilde{Z}_t,\;\;\; t\in[0,T_0],\\\\
\widehat{Z}_t,\;\;\; t\in[T_0,\infty],
\end{array}\right.
\;\;\;K_t=
\left\{
\begin{array}{l}
\widetilde{K}_t,\;\;\; t\in[0,T_0]\\\\
\widetilde{K}_{T_0}+\widehat{K}_t-\widehat{K}_{T_0},\;\;\; t\in[T_0,\infty].
\end{array}\right.
\end{eqnarray*}
If $t\in[T_0,\infty],\; (\widehat{Y}_t,\widehat{Z}_t,\widehat{K}_t)$ is the solution of $(\ref{BSDEinfh1})$, and then $(\widehat{Y}_t,\widehat{Z}_t,\widetilde{K}_{T_0}+\widehat{K}_t-\widehat{K}_{T_0})$ also satisfies $(\ref{BSDEinfh1})$. Now, if $t\in[0, T_0]$ ,\; $(\widetilde{Y}_t,\widetilde{Z}_t,\widetilde{K}_t)$ is the solution of $(\ref{BSDEinfh2})$ and $\widetilde{Y}_{T_0} = \widehat{Y}_{T_0},\; \widetilde{K}_{T_0} = \widetilde{K}_{T_0}+\widehat{K}_{T_0}-\widehat{K}_{T_0}$. So $Y$ and $K$ are continuous, and $(Y, Z, K)$ is a unique solution of reflected GBSDE $(\ref{BSDEinfh})$.
\end{proof}
\begin{remark}
If the random variable $\xi\equiv 0$ a.s,  the condition $({\bf A3})$ remain true and Theorem $4.2$  is available with assumptions  $({\bf A1})$-$({\bf A4})$.
The proof follows steps of proof of Theorem $3.1$ taking $\tau=\infty$
\end{remark}
\section{Applications}

In this section, we consider reflected GBSDEs in Markovian framework and
stated is related to an American option pricing as well as is related to a
probabilistic representation of the viscosity solution of an obstacle
problem of elliptic type.

\subsection{A class of reflected diffusion process}

Let $b:{\Bbb R}^{d}\longrightarrow {\Bbb R}^{d},$ $\sigma :{\Bbb R}%
^{d}\longrightarrow {\Bbb R}^{d\times d}$ be functions such that
\[
\left| b\left( x\right) -b\left( x^{\prime }\right) \right| +\left| \sigma
\left( x\right) -\sigma \left( x^{\prime }\right) \right| \leq K\left|
x-x^{\prime }\right| .
\]
Let $\Theta $ be an open connected bounded subset of ${\Bbb
R}^{d},$ which is that for a function $\phi \in {\cal
C}_{b}^{2}({\Bbb R}^{d}),\Theta =\left\{ \phi >0\right\} ,$
$\partial \Theta =\left\{ \phi =0\right\} ,$ and $\left|
\bigtriangledown \phi \left( x\right) \right| =1,$ $x\in \partial
\Theta .$ Note that at any boundary point $x\in \partial \Theta ,$ $%
\bigtriangledown \phi \left( x\right) $ is a unit normal vector to
the boundary, pointing towards the interior of $\partial \Theta
.$\newline By Lions and Szitman \cite{LS} (see also Saisho
\cite{S}) for each $x\in \overline{\Theta }$ there exists a unique
pair of progressively measurable continuous processes $\left\{
(X_{s}^{x},G_{s}^{x}):t\geq 0\right\} $, with values in
$\overline{\Theta }\times {\Bbb R}_{+,}$ such that

\allowdisplaybreaks\begin{align}
s\mapsto &G_{s}^{x} \,\,\,\hbox{is increasing},\nonumber\\
X_{s}^{x}&=x+\int_{0}^{s}b(X_{r}^{x})dr+\int_{0}^{s}\sigma
(X_{r}^{x})dW_{r}+\int_{0}^{s}\nabla \phi (X_{r}^{x})dG_{r}^{x},\;\; s\geq 0, \nonumber\\
G_{s}^{x}&=\int_{0}^{s}1_{\left\{ X_{r}^{x}\in \partial \Theta
\right\} }dG_{r}^{x}.
  \label{k*}
\end{align}

Let state some properties of processes $\left\{ (X_{s}^{x},G_{s}^{x}),s\geq
0\right\} .$ We refer the reader to Pardoux and Zhang, \cite{PZ}.

\begin{proposition}
For each $T\geq 0,$ there exits a constant $C_T$ such that for all $%
x,x^{^{\prime }}\in \overline{\Theta }$

\begin{eqnarray*}
\E\left(\sup_{0\leq s\leq
T}|X_{s}^{x}-X_{s}^{x^{^{\prime
}}}|^{4}\right)\leq C_T|x-x^{^{\prime }}|^{4}
\end{eqnarray*}
and
\begin{eqnarray*}
\E\left(\sup_{0\leq s\leq
T}|G_{s}^{x}-G_{s}^{x^{^{\prime }}}|^{4}\right)\leq C_T|x-x^{^{\prime
}}|^{4}.
\end{eqnarray*}
Moreover, there exists a constant $C_p$ such that for all $(t,x)\in {\Bbb R}%
_{+}\times \overline{\Theta }$,
\begin{eqnarray*}
\E(|G_{t}^{x}|^{p})\leq C_p\left(
1+t ^{p}\right),
\end{eqnarray*}
and for each $\mu ,t>0,$ there exists $C_{\mu,t}$ such that
for all $x\in \overline{\Theta }$,
\begin{eqnarray*}
\E\left( e^{\mu G_{t}^{x}}\right)\leq C_{\mu,t}.
\end{eqnarray*}
\end{proposition}

Since we state in Markovian framework, the $(\xi ,f,g,S)$ are defined as follows:
\[f(s,y,z)=f(s,X_{s}^{x},y,z),\mbox{ }
g(s,y)=g(s,X_{s}^{x},y), \mbox{ }S_{s}=h(X_{s}^{x}),
\]
where $f$, $g$ satisfy the previous assumptions as we have in random finite horizon or infinite horizon and $h\in {\cal C}(\R^{d};\R)$ with most polynomial growth at infinity.

\subsection{American option pricing revisited}
In this section, we use the result on infinite horizon reflected GBSDEs with one barrier
to deal with optimal stopping time problem. Roughly speaking, let us consider the following reflected GBSDE:
\begin{enumerate}
\item
\begin{eqnarray}
Y_{s}^{x}&=&\xi+
\int_{s}^{\infty}f(r,X_{r}^{x},Y_{r}^{x},Z_{r}^{x})dr+%
\int_{s}^{\infty}g(r,X_{r}^{x},Y_{r}^{x})dG_{r}^{x}\nonumber\\
&&-\int_{s}^{\infty}Z_{r}^{x}dW_{r}+K_{\infty}^{x}-K_{s}^{x},\; 0\leq s\leq \infty,\label{C1}
\end{eqnarray}
\item $Y_{s}^{x}\geq h(X_{s}^{x})$,
\item $\E\left(\sup_{0\leq t\leq \infty}|Y^{x}_t|^{2}+\int_{0}^{\infty}\left|Z_{r}^{x}\right|^{2}dr\right) <+\infty$,
\item \, $K_{s}^{x}$ is an increasing process such that $K_{0}=0$ and $\int_{0}^{\infty}(Y_{s}^{x}-h(X_{s}^{x}))dK_{s}^{x}=0$.
\end{enumerate}
From Theorem 4.1, the previous reflected GBSDE has a unique solution $\left(Y^{x},Z^{x},K^{x}\right)$.
Unlike of the work of Cvitanic and Ma, \cite{CM}, we interpret $X^x$ in \eqref{k*} as a price process of financial assets which might affect the wealth of a controller and forced to live in a bounded domain; $Y^x$ and $Z^x$ are the wealth process and the trading strategy, respectively, of a "small" investor or a "small" shareholder in the market in the sense that both $Y^x$ and $Z^x$ might no affect the price $X^x$. The investor acts to protect his advantages so that he has possibility at any time $\theta\in\mathcal{K}$ (set of all $\mathcal{F}_s$-stopping time with values in $[0,\infty]$) to stop controlling. The control is not free. We define the pay off  by
\begin{eqnarray*}
R(\theta) &=&\E\left\{\int_{0}^{\theta
}f(r,X_{r}^{x},Y_{r}^{x},Z_{r}^{x})dr+\int_{0}^{\theta
}g(r,X_{r}^{x},Y_{r}^{x})dG_{r}^{x}\right. \\
&& \left. +h(X_{\theta }^{x}){\bf 1}_{\left\{ \theta <\infty\right\}
}+\xi{\bf 1}_{\left\{\theta =\infty\right\} }\right\}
\end{eqnarray*}
for all $\theta\in\mathcal{K}$. For the investor, $f(X^x,Y^x,Z^x)$, (resp. $f(X^x,Y^x,Z^x)+g(X^x,Y^x)\dot{G}^x$) is the instantaneous reward on $\Theta$ (resp. on $\partial\Theta$), and $h(X^{x})$ and $\xi$ are respectively the rewards if he decides to stop before or until infinite time.
The problem is to look for an optimal strategy for the investor, i.e. a strategy $\widehat{\theta}$ such that
\begin{eqnarray*}
R(\theta)\leq R(\widehat{\theta})\;\;\;\mbox{for all}\;\;  \theta\in\mathcal{K}.
\end{eqnarray*}

Now we give the main result of this section, an analogue of that in Cvitanic and Ma, \cite{CM}.
\begin{theorem}
Let $(Y_{.}^{x},Z_{.}^{x},K_{.}^{x})$ be a unique solution of reflected
GBSDE $(\ref{C1})$. Then there exists an
optimal stopping time given by
\begin{eqnarray*}
\widehat{\theta}=\
\left\{
\begin{array}{l}
inf \left\{ t\in [0,\infty),\;\;Y_{t}^{x}\leq h(X_{t}^{x})\right\},\\\\
\infty\;\;\;\; otherwise.
\end{array}\right.
\end{eqnarray*}
Then $Y^{x}_0$ = $R(\widehat{\theta})$, and $\widehat{\theta}$ is an optimal strategy for the investor.
\end{theorem}
\begin{proof}
Since $(Y^{x}, Z^x, K^x)$ is a unique solution of reflected
GBSDE $(\ref{C1})$, $Y^x_0$ is deterministic and we have
\begin{eqnarray}
Y_{0}^{x}=\E(Y^x_0)&=&\E\left(\xi+\int_{0}^{\infty}f(X_{r}^{x},Y_{r}^{x},Z_{r}^{x})dr+\int_{0}^{\infty}g(r,X_{r}^{x},Y_{r}^{x})dG_{r}^{x}\right.\nonumber\\
&&\left.-\int_{0}^{\infty}Z_{r}^{x}dW_{r}+K_{\infty}^{x}\right)\nonumber\\
&=&\E\left(Y^{x}_{\widehat{\theta}}+\int_{0}^{\widehat{\theta}}f(X_{r}^{x},Y_{r}^{x},Z_{r}^{x})dr+\int_{0}^{\widehat{\theta}}g(r,X_{r}^{x},Y_{r}^{x})dG_{r}^{x}\right.\nonumber\\
&&\left.-\int_{0}^{\widehat{\theta}}Z_{r}^{x}dW_{r}+K_{\widehat{\theta}}^{x}\right)
\label{C2}
\end{eqnarray}
In view of $\widehat{\theta}$ and reflected GBSDE's properties one knows that
the process $K_t$ does not increase between $0$ and $\widehat{\theta}$, hence then $K_{\widehat{\theta}}= 0$.

On the other hand, since $\int_{0}^{\widehat{\theta}}Z_{r}^{x}dW_{r}$ is a martingale, we get
\begin{eqnarray*}
Y_{0}^{x}=\E\left(Y^{x}_{\widehat{\theta}}+\int_{0}^{\widehat{\theta}}f(X_{r}^{x},Y_{r}^{x},Z_{r}^{x})dr
+\int_{0}^{\widehat{\theta}}g(r,X_{r}^{x},Y_{r}^{x})dG_{r}^{x}\right).
\end{eqnarray*}
Next, $\displaystyle{Y_{\widehat{\theta}}^{x}=h(X_{\widehat{\theta_{t}}}^{x}){\bf 1}_{\left\{ \widehat{\theta}<\infty\right\}
}+\xi{\bf 1}_{\left\{ \widehat{\theta}=\infty\right\}}}$ a.s., implies $\displaystyle{Y^{x}_0=R(\widehat{\theta})}$.

Now from (\ref{C2}), we deduce that for every $\theta \in {\cal K},$%
\begin{eqnarray*}
Y_{0}^{x} &=&\E\left\{ Y_{\theta }^{x}+\int_{0}^{\theta
}f(r,X_{r}^{x},Y_{r}^{x},Z_{r}^{x})dr\right. \\
&&\left. +\int_{0}^{\theta
}g(r,X_{r}^{x},Y_{r}^{x})dG_{r}^{x}+K_{\theta }^{x}\right\} .
\end{eqnarray*}
But $K_{\theta}^{x}\geq 0$ and $\displaystyle{Y_{\theta }^{x} \geq h(X_{\theta }^{x}){\bf 1}_{\left\{ \theta <\infty\right\}
}+\xi{\bf 1}_{\left\{ \theta =\infty\right\}}}$.
Then,
\begin{eqnarray*}
R(\widehat{\theta})=Y_{0}^{x} &\geq &\E\left\{\int_{0}^{\theta
}f(r,X_{r}^{x},Y_{r}^{x},Z_{r}^{x})dr+\int_{0}^{\theta}g(r,X_{r}^{x},Y_{r}^{x})dG_{r}^{x}+h(X_{\theta}^{x}){\bf 1}_{\left\{\theta <\infty\right\}}+\xi{\bf 1}_{\left\{ \theta =\infty\right\}}\right\}\\
&\geq& R(\theta).
\end{eqnarray*}
Hence the stopping time $\widehat{\theta}$ is optimal.
\end{proof}

\subsection{An obstacle problem for elliptic PDEs with nonlinear Neumann
boundary condition}

In this subsection, we will show that in the Markovian case the solution of the reflected GBSDEs with random terminal time
is a solution of an obstacle problem for elliptic PDEs with a nonlinear Neumann
boundary condition. It follows from the results of the Section 3 that for all $x\in\overline{\Theta}$, there exists a unique triple $\left(Y^{x},Z^{x},K^{x}\right)$ be the unique
solution of the following reflected GBSDE:
\begin{enumerate}
\item
\begin{eqnarray}
Y_{s}^{x}&=&h(X^{x}_{\tau})+
\int_{s}^{\tau}f(r,X_{r}^{x},Y_{r}^{x},Z_{r}^{x})dr+%
\int_{s}^{\tau}g(r,X_{r}^{x},Y_{r}^{x})dG_{r}^{x}\nonumber\\
&&-\int_{s}^{\tau}Z_{r}^{x}dW_{r}+K_{\tau}^{x}-K_{s}^{x},\; 0\leq s\leq \tau,\label{C1}
\end{eqnarray}
\item $Y_{s}^{x}\geq h(X_{s}^{x})$,
\item $\E\left(\sup_{0\leq t\leq \tau}|Y^{x}_t|^{2}+\int_{0}^{\tau}\left|Z_{r}^{x}\right|^{2}dr\right) <+\infty$,
\item \, $K_{s}^{x}$ is an increasing process such that $K_{0}=0$ and $\int_{0}^{\tau}(Y_{s}^{x}-h(X_{s}^{x}))dK_{s}^{x}=0$.
\end{enumerate}

We now consider the related obstacle problem for elliptic PDEs with a
nonlinear Neumann boundary condition. Roughly speaking, a solution of the obstacle problem is a function $u\in C(\overline{\Theta};\R)$ which
satisfies:
\begin{align}
&&\min\left\{u\left(x\right)-h(x),Lu(x)+f(x,\,u(x),(\nabla u)^{*}\sigma(x))\right\}=0, \;\;x\in\Theta, \nonumber\\
\label{A1}\\
&&\frac{\partial u}{\partial n}(x)+g(x,u(x)) =0,\;\;x\in\partial \Theta,
\nonumber
\end{align}
where
\[
L=\frac{1}{2}\sum_{i,j=1}^{d}\left(
\sigma \sigma ^{*}\right) _{ij}\left(x\right) \frac{\partial ^{2}}{\partial x_{i}\partial x_{j}}
+\sum_{i=1}^{d}b_{i}\left(x\right) \frac{\partial}{\partial x_{i}
}
\]
and at point $x\in \partial \Theta $%
\[
\frac{\partial}{\partial n}=\sum_{i=1}^{d}\frac{\partial \psi }{
\partial x_{i}}\left( x\right) \frac{\partial}{\partial x_{i}}.
\]
More precisely, solutions of Equation $(\ref{A1})$ is take in viscosity sense.

\begin{definition}
$(a)$ $u\in {\cal C}\left(\overline{\Theta },\R^{d}\right)$ is said to be a viscosity subsolution of (\ref{A1}) if
for any point $x_0\in\overline{\Theta }$, such that $u(x_0)>h(x_0)$ and for any $\varphi\in C^{2}(\overline{\Theta})$ such that $\varphi(x_{0})=u(x_0)$ and $u-\varphi$ attains its minimum at  $x_0$, then
\begin{eqnarray}
\begin{array}{l}
-Lu(x_0)-f(x,u(x_0),(\nabla u\sigma)(x_0))\leq 0,\;\mbox{if }\;x_0\in \Theta
\\\\
\min\left(-Lu(x_0)-f(x,u(x_0),(\nabla u\sigma)(x_0)),\;-\frac{\partial \varphi}{\partial n}(x_0)-g(x_0,\;-\varphi(x_0))\right)\leq 0,\;\mbox{if }\; x\in \partial \Theta.
\end{array}
\end{eqnarray}

$(b)$ $u\in {\cal C}\left(\overline{\Theta },\R^{d}\right)$ is said to be a viscosity supersolution of (\ref{A1}) if
for any point $x_0\in\overline{\Theta }$, such that $u(x_0)\geq h(x_0)$ and for any $\varphi\in C^{2}(\overline{\Theta})$ such that $\varphi(x_{0})=u(x_0)$ and $u-\varphi$ attains its maximum at  $x_0$, then
\begin{eqnarray}
\begin{array}{l}
-Lu(x_0)-f(x,u(x_0),(\nabla u\sigma)(x_0))\geq 0,\;\mbox{if }\;x_0\in \Theta
\\\\
\min\left(-Lu(x_0)-f(x,u(x_0),(\nabla u\sigma)(x_0)),\;-\frac{\partial \varphi}{\partial n}(x_0)-g(x_0,\varphi(x_0))\right)\geq 0,\;\mbox{if }\; x\in \partial \Theta.
\end{array}
\end{eqnarray}
$(c)$ $u$ is a viscosity solution of (\ref{A1}) if it is both a viscosity subsolution and supersolution.
\end{definition}
We define
\begin{eqnarray}
u\left(x\right) =Y_{0}^{x},\;\; x\in\overline{\Theta }  \label{E1}
\end{eqnarray}
which is a deterministic quantity since $Y_{0}^{x}$ is measurable with
respect to the $\sigma$-algebra $\sigma \left( W_{r}:0\leq r\leq
\infty\right) .$ For standards estimates for reflected GBSDEs and Proposition 4.1, we deduce

\begin{proposition}
The function $u\in C(\Theta;\R)$ such that \,\,$u(x) \geq h(x) \;\; \forall \mbox{ }x\in\overline{\Theta }$
\end{proposition}

The main result in this subsection is the following.

\begin{theorem}
The function defined by $(\ref{E1})$ is a viscosity solution of $(\ref{A1})$.
\end{theorem}

\begin{proof}
First, let us show that $u$ is a viscosity subsolution of $(\ref{A1})$. Let $x_0\in\overline{\Theta}$ and $\varphi\in C^{2}(\overline{\Theta};\R^d)$ be such
that $\varphi(x_0)= u(x_0)$ and $\varphi(x_0)\geq u(x)$ for all $x\in\overline{\Theta}$.

Step 1: Suppose that $u(x_0)> h(x_0)$ and $x_0\in\Theta$ and
\begin{eqnarray*}
-L\varphi(x_0)-f(x,\varphi(x_0),(\nabla \varphi\sigma)(x_0))>0,
\end{eqnarray*}
and we will find a contradiction.

Indeed, by continuity, we can suppose that there exist $\varepsilon> 0$ and $\eta_{\varepsilon}> 0$ such
that for each $x\in\{y: |y-x_0|<\eta_{\varepsilon}\subset\Theta$, we have
$u(x)\geq h(x)+\varepsilon$ and
\begin{eqnarray}
-Lu(x)-f(x,\varphi(x),(\nabla \varphi\sigma)(x))\geq \varepsilon. \label{viscosity1}
\end{eqnarray}
Define
\begin{eqnarray}
\tau = \inf\left\{s\geq 0:\;  |X^{x_0}_s-x_0|>\eta_{\varepsilon}\right\}\label{viscosity2}
\end{eqnarray}
Note that, for all $s\in[0,\infty]$
\begin{eqnarray*}
 u(X^{x_0}_s)\geq h(X^{x_0}_s)+\varepsilon.
\end{eqnarray*}
Consequently, the process $K^{x_0}_s$ is constant on $[0,\tau]$ and, hence,
\begin{eqnarray*}
Y_{s}^{x}&=&Y^{x_0}_{\tau}+\int_{s}^{\tau}f(X_{r}^{x_0},Y_{r}^{x_0},Z_{r}^{x_0})dr-\int_{s}^{\tau}Z_{r}^{x_0}dW_{r}, \; 0\leq s\leq \tau.
\end{eqnarray*}
On the other hand, applying Itô's formula to $\varphi(X^{x_0}_s)$ gives
\begin{eqnarray*}
\varphi(X^{x_0}_s)&=&\varphi(X^{x_0}_\tau)-\int_{s}^{\tau}L\varphi(X^{x_0}_r)dr-\int_{s}^{\tau}\nabla\varphi\sigma(X_{r}^{x_0})dW_{r},\; 0\leq s\leq \tau.
\end{eqnarray*}
Now, by inequality $(\ref{viscosity1})$,
\begin{eqnarray*}
-L\varphi(X^{x_0}_s)-f(X^{x_0}_s,\varphi(X^{x_0}_s),(\nabla \varphi\sigma)(X^{x_0}_s))\geq \varepsilon.
\end{eqnarray*}
Also,
\begin{eqnarray*}
\varphi(X^{x_0}_\tau)\geq u(X^{x_{0}}_\tau)=Y^{x_0}_\tau.
\end{eqnarray*}
Consequently, comparison theorem for GBSDEs (see \cite{PZ}) implies
\begin{eqnarray*}
\varphi(x_0)>\varphi(X^{x_0}_\tau)-\tau\varepsilon\geq u(x_0),
\end{eqnarray*}
which leads to a contradictions.

Step 2: If we further suppose that $u(x_0)> h(x_0)$ and $x_0\in\partial\Theta$ and
\begin{eqnarray}
\min\left(-L\varphi(x_0)-f(x,\varphi(x_0),(\nabla \varphi\sigma)(x_0)),\; -\frac{\partial\varphi}{\partial n}-g(x_0,\varphi(x_0))\right)>0.\label{viscosity3}
\end{eqnarray}
By continuity, we can suppose that there exist $\varepsilon> 0$ and $\eta_{\varepsilon}> 0$ such
that for each $x\in\{y: |y-x_0|<\eta_{\varepsilon}\subset\Theta$, we have
$u(x)\geq h(x)+\varepsilon$ and
\begin{eqnarray}
\min\left(-Lu(x)-f(x,\varphi(x),(\nabla \varphi\sigma)(x)),\, -\frac{\partial\varphi}{\partial n}-g(x,\varphi(x))\right)\geq \varepsilon. \label{viscosity4}
\end{eqnarray}
Let $\tau$ be the stopping time defined as above by $(\ref{viscosity2})$ and note that, for all $s\in[0,\tau]$
\begin{eqnarray*}
 u(X^{x_0}_s)\geq h(X^{x_0}_s)+\varepsilon.
\end{eqnarray*}
Consequently, the process $K^{x_0}_s$ is constant on $[0,\tau]$ and, hence,
\begin{eqnarray*}
Y_{s}^{x}&=&Y^{x_0}_{\tau}+\int_{s}^{\tau}f(X_{r}^{x_0},Y_{r}^{x_0},Z_{r}^{x_0})dr+\int_{s}^{\tau}g(r,X_{r}^{x_0},Y_{r}^{x_0})dG_{r}^{x_0}\nonumber\\
&&-\int_{s}^{\tau}Z_{r}^{x_0}dW_{r}, \; 0\leq s\leq \tau.
\end{eqnarray*}
On the other hand, applying Itô's formula to $\varphi(X^{x_0}_s)$ gives
\begin{eqnarray*}
\varphi(X^{x_0}_s)&=&\varphi(X^{x_0}_\tau)-\int_{s}^{\tau}L\varphi(X^{x_0}_r)dr-\int_{s}^{\tau}\frac{\partial{\varphi}}{\partial n}(X^{x_0}_r)dG^{x_0}_r-\int_{s}^{\tau}\nabla\varphi\sigma(X_{r}^{x_0})dW_{r},\; 0\leq s\leq \tau.
\end{eqnarray*}
Now, by $(\ref{viscosity4})$,
\begin{eqnarray*}
\min\left(-L\varphi(X^{x_0}_s)-f(X^{x_0}_s,\varphi(X^{x_0}_s),(\nabla \varphi\sigma)(X^{x_0}_s)),\; -\frac{\partial{\varphi}}{\partial n}(X^{x_0}_s)-g(r,X_{r}^{x_0},Y_{r}^{x_0})\right)\geq \varepsilon.
\end{eqnarray*}
Also,
\begin{eqnarray*}
\varphi(X^{x_0}_\tau)\geq u(X^{x_{0}}_\tau)=Y^{x_0}_\tau.
\end{eqnarray*}
Consequently, comparison theorem for GBSDEs (see \cite{PZ}) implies
\begin{eqnarray*}
\varphi(x_0)>\varphi(X^{x_0}_\tau)-\tau\varepsilon\geq u(x_0),
\end{eqnarray*}
which leads to a contradiction.

By the same argument as above one can show that $u$ given by $(\ref{E1})$ is also a viscosity supersolution of elliptic refected PDEs $(\ref{A1})$ and ends the proof.
\end{proof}

\end{document}